\documentclass[a4paper,12pt]{article}
\usepackage{amssymb}
\usepackage{amsthm}
\usepackage{tikz}
\usepackage{tikz,pgfplots}
\usetikzlibrary{decorations.markings}
\usepackage{amsmath,amssymb}
\usepackage{todonotes}
\usepackage{url}

\usepackage{color}
\usepackage{graphicx}
\usepackage{placeins,caption}
\newcommand{\address}[1]{#1}

\DeclareMathOperator{\gp}{gp}
\DeclareMathOperator{\diam}{diam}

\DeclareMathOperator{\tp}{tp}
\DeclareMathOperator{\g}{g}

\DeclareMathOperator{\cp}{\,\square\,}
\usepackage[top=3cm,bottom=3cm,left=2.8cm,right=2.8cm]{geometry}
\newtheorem{theorem}{Theorem}[section]
\newtheorem{lemma}[theorem]{Lemma}
\newtheorem{corollary}[theorem]{Corollary}
\newtheorem{proposition}[theorem]{Proposition}

\newtheorem{conjecture}[theorem]{Conjecture}

\theoremstyle{definition}

\newtheorem{definition}[theorem]{Definition}

\tikzset{middlearrow/.style={
		decoration={markings,
			mark= at position 0.75 with {\arrow[scale=2]{#1}} ,
		},
		postaction={decorate}
	}
}

\tikzset{midarrow/.style={
		decoration={markings,
			mark= at position 0.75 with {\arrow[scale=2]{#1}} ,
		},
		postaction={decorate}
	}
}

\begin{document}
	
	\title{Lower General Position in Cartesian Products}
	\author{Eartha Kruft Welton $^{a,b}$ \\ \texttt{\footnotesize kruftweltonei@cardiff.ac.uk}
		\and
		Sharif Khudairi $^{b}$ \\ \texttt{\footnotesize
			khudairis@cardiff.ac.uk}
		\and
		James Tuite $^{a}$ \\ \texttt{\footnotesize james.t.tuite@open.ac.uk}        
	}

	\maketitle	
	
	\address{
		\noindent	
		$^a$ School of Mathematics and Statistics, Open University, Milton Keynes, UK \\
		$^b$ School of Mathematics, University of Cardiff, UK 
	}
	
	\begin{abstract}
		A subset $S$ of vertices of a graph $G$ is in \emph{general position} if no shortest path in $G$ contains three vertices of $S$. The \emph{general position problem} consists of finding the number of vertices in a largest general position set of $G$, whilst the \emph{lower general position problem} asks for a smallest maximal general position set. In this paper we determine the lower general position numbers of several families of Cartesian products. We also show that the existence of small maximal general position sets in a Cartesian product is connected to a special type of general position set in the factors, which we call a \emph{terminal set}, for which adding any vertex $u$ from outside the set creates three vertices in a line with $u$ as an endpoint. We give a constructive proof of the existence of terminal sets for graphs with diameter at most three. We also present conjectures on the existence of terminal sets for all graphs and a lower bound on the lower general position number of a Cartesian product in terms of the lower general position numbers of its factors.
	\end{abstract}
	
	
	\noindent
	AMS Subj.\ Class.\ (2020): 05C12, 05C69, 05C76
	
	\section{Introduction}\label{sec:intro}
	
	The lower general position problem originates in the puzzles of two of the best known recreational mathematicians of modern history, Henry Dudeney and Martin Gardner. Dudeney posed the following chessboard puzzle in~\cite{dudeney-1917}: how many pawns can be placed on an $n \times n$ chessboard if we do not allow any three pawns to lie on a common line in the plane? In Gardner's column in Scientific American~\cite{Gar}, he suggested investigating the `worst-case scenario' of this problem by finding the smallest no-three-in-line configurations of pawns on the chessboard that cannot be extended by adding another pawn.
	
	The no-three-in-line problem was extended to the setting of graph theory in~\cite{ullas-2016,manuel-2018}. A set $S$ of vertices of a graph $G$ is in \emph{general position} if no shortest path of $G$ contains more than two vertices of $S$. The \emph{general position number} $\gp (G)$ is the cardinality of a largest general position set of $G$. Inspired by Gardner's suggestion, a recent paper~\cite{SteKlaKriTuiYer} studied the \emph{lower position number} $\gp ^-(G)$ of a graph, which the authors defined to be the number of vertices in a smallest maximal general position set of $G$. The lower general position number can be seen as representing the worst-case output of a greedy search for general position sets. 
	
	There are several papers on the general position numbers of the Cartesian product of graphs. The first paper to treat this problem was~\cite{Ghorbani-2021}; in this article the lower bound $\gp (G \cp H) \geq \gp (G)+\gp (H)-2$ is deduced and this is used to find a lower bound on the general position number of the product of an arbitrary number of complete graphs. In~\cite{KlaPatRusYer} the authors determine the exact values of the general position numbers of cylinder graphs $P_r \cp C_s$ and strong bounds for torus graphs $C_r \cp C_s$, as well as counting all maximum general position sets in a Cartesian grid $P_r \cp P_s$ and giving lower bounds for Cartesian powers. The paper~\cite{TianXu} discusses the case when one of the factors has a small diameter and in~\cite{TianXuKlav} the general position number of the product of two trees is determined. The latest work on this subject~\cite{KorVes} largely resolves the general position problem in torus graphs and gives some partial results on hypercubes.
	
	More recently, variants of the general position problem have been investigated for Cartesian products; they are treated in the context of the mutual visibility problem in~\cite{CicSteKla} and~\cite{KorVes2}, the edge general position problem in~\cite{ManPraKla} and the monophonic position problem in~\cite{NeeChaTui} (for background on the mutual visibility and monophonic position problems, consult~\cite{Stef} and~\cite{ThoChaTui} respectively).
	
	In this paper we study the lower general position problem in Cartesian products. We demonstrate that this problem is connected with a special type of general position set in a graph, which we call a \emph{terminal set}. In Section~\ref{sec:bounds} we present a bound on the lower general position number of Cartesian products using terminal sets and use this to derive some exact values. We also discuss products with complete graphs and give a conjecture for a lower bound on the lower general position number of a Cartesian product. In Section~\ref{sec:existence} we tackle the problem of the existence of terminal sets and show that they always exist in graphs with diameters at most three using an algorithmic proof. We also demonstrate existence for chordal graphs and cographs and conclude with a conjecture on the existence of terminal sets. Finally in Section~\ref{sec:terminal pos num} we derive some exact values for the largest and smallest order of terminal position sets in some common graph families and apply this to find the lower general position number of the product of complete multipartite graphs.   
	
	Throughout this paper a graph $G$ will mean a simple, undirected graph, with vertex set $V(G)$ and edge set $E(G)$. We will write $u \sim v$ if the vertices $u,v \in V(G)$ are adjacent. The neighbourhood of $u$ is the set $N(u) = \{ v \in V(G): u \sim v\} $. Two vertices $u,v$ are \emph{twins} if they have the same neighbourhood, i.e. $N(u) = N(v)$.  A graph $H$ is a subgraph of $G$ if $V(H) \subseteq V(G)$ and $E(H) \subseteq E(G)$; a subgraph $H$ of $G$ is \emph{induced} if for any $u,v \in V(H)$ we have $u \sim v$ in $H$ if and only if $u \sim v$ in $G$. For any subset $W \subseteq V(G)$ the subgraph of $G$ induced by $W$ will be written $G[W]$. The \emph{line graph} $L(G)$ of $G$ is the graph with vertex set equal to $E(G)$, with an edge in $L(G)$ between $e_1,e_2 \in E(G)$ if and only if $e_1$ and $e_2$ are incident to a common vertex.  
	
	A path $P_{\ell +1}$ of length $\ell $ in $G$ is a sequence $u_0,u_1,\dots ,u_{\ell }$ of distinct vertices of $G$ such that $u_i \sim u_{i+1}$ for $0 \leq i \leq \ell -1$. The distance $d(u,v)$ between two vertices $u$ and $v$ of $G$ is the length of a shortest path from $u$ to $v$; if we wish to specify in which graph the distance is taken we will include a subscript, e.g. $d_G(u,v)$. The greatest value of $d(u,v)$ over all pairs of vertices $u,v$ in $G$ is the \emph{diameter} $\diam (G)$ of $G$. A cycle of length $\ell $ is a sequence $u_0,u_1,\dots ,u_{\ell -1}$ such that $u_i \sim u_{i+1}$ for $1 \leq i \leq \ell -2$ and $u_0 \sim u_{\ell -1}$. The wheel graph $W_n$ is the join $C_{n-1} \vee K_1$, i.e. a cycle of length $n-1$ with an added vertex that is adjacent to every vertex of the cycle. A graph is \emph{chordal} if all of its induced cycles have length three (equivalently, any cycle with length at least four has a chord). A graph is a \emph{cograph} if it does not contain any induced copies of $P_4$.
	
	A set $S$ of vertices of $G$ is \emph{geodetic} if for any vertex $u \in V(G) - S$ there exist vertices $x,y \in S$ such that $u$ lies on a shortest $x,y$-path in $G$; the number of vertices in a smallest geodetic set of $G$ is the \emph{geodetic number} of $G$ and is denoted by $\g (G)$. A clique or complete graph $K_n$ is a graph on $n$ vertices such that every pair of distinct vertices is adjacent. More generally, a complete multipartite graph $K_{r_1,r_2,\dots ,r_t}$ is a graph such that the vertex set can be partitioned into independent sets $X_1, X_2, \dots ,X_t$, where $|X_i| = r_i$ for $1 \leq i \leq t$, such that $u \sim v$ if and only if $u$ and $v$ lie in different parts of the partition.
	
	The Cartesian product $G \cp H$ of two graphs $G$ and $H$ is defined to be the graph with vertex set $V(G) \times V(H)$ and an edge between vertices $(u_1,v_1)$ and $(u_2,v_2)$ if and only if either i) $u_1 = u_2$ and $v_1 \sim v_2$ in $H$ or ii) $u_1 \sim u_2$ in $G$ and $v_1 = v_2$. For any $h \in V(H)$, the subgraph of $G \cp H$ induced by $V(G) \times \{ h\} $ is isomorphic to $G$; we call this subgraph a \emph{$G$-layer} of $G \cp H$ and denote it by $G^h$. Similarly, a $H$-layer $^gH$ is the subgraph of $G \cp H$ induced by the subset $\{ g\} \times V(H)$ of $V(G) \times V(H)$ and is isomorphic to $H$. If $P$ is a path $u_0,u_1,\dots ,u_{\ell }$ in $G$, then for any $v \in V(H)$ the path $(u_0,v),(u_1,v),\dots ,(u_{\ell },v)$ in $G \cp H$ will be written $P_v$, with a similar notation $_uQ$ for $u \in V(G)$ and a path $Q$ in $H$. The operation $\cp $ is associative and commutative (up to isomorphism), so in Cartesian products of three or more factors we will drop brackets. The distance between vertices $(u,v),(u',v') \in V(G \cp H)$ is given by \[ d_{G \cp H}((u,v),(u',v')) = d_G(u,u')+d_H(v,v').\]  We define the projection functions $\pi _1: G \cp H \rightarrow G$ and $\pi _2: G \cp H \rightarrow H$ by $\pi _1(x,y) = x$ and $\pi _2 (x,y) = y$ for all $(x,y) \in V(G) \times V(H)$. For a detailed discussion of the structure of Cartesian products, see the book~\cite{ikr}. For other graph-theoretical terminology not defined here, we refer the reader to~\cite{BonMur}.   
	
	\section{Bounds and exact values}\label{sec:bounds}
	
	In this section we introduce the notion of \emph{terminal set} and its applications to finding small general position sets in Cartesian products. It was shown in~\cite{SteKlaKriTuiYer} that a graph $G$ has a maximal general position set of order two, i.e. $\gp ^-(G) = 2$, if and only if $G$ contains a \emph{universal line}. The line ${\cal L}(u,v)$ of $G$ induced by two vertices $u,v \in V(G)$ is defined to be \[ \{w\in V(G):\ d(u,v) = d(u,w) + d(w,v)\ {\rm or}\ d(u,v) = |d(u,w) - d(w,v)|\} .\]
	The line ${\cal L}(u,v)$ is \emph{universal} if it contains every vertex of $G$. More generally, universal lines are of interest in the setting of metric spaces and are connected to the Chen-Chv\'atal Conjecture (see~\cite{ChenChv,rodriguez-2022}).
	
	The article~\cite{rodriguez-2022} gave the following necessary and sufficient condition for the existence of a universal line in a Cartesian product.
	
	\begin{theorem}\label{thm:cp}
		For two graphs $G$ and $H$, the Cartesian product $G \cp H$ contains a universal line if and only if either
		\begin{enumerate}
			\item[(i)] $G$ or $H$ has a maximal general position set consisting of two adjacent vertices, or
			\item[(ii)] $\g(G) = 2$ and $\g(H) = 2$.
		\end{enumerate}
	\end{theorem}
	As shown in~\cite{SteKlaKriTuiYer} and~\cite{rodriguez-2022}, any bipartite graph or any graph with a bridge contains two adjacent vertices that form a universal line; hence any Cartesian product $G \cp H$ containing such a factor $G$ or $H$ will satisfy the conditions of Theorem~\ref{thm:cp}, implying that $\gp ^-(G \cp H) = 2$. In particular, if either of the factors of $G \cp H$ is a tree or an even cycle, the product will contain a universal line. Hence for many of the families of Cartesian product graphs discussed in the preceding section the lower general position number follows easily from known results; hence for interesting problems we will have to look to more complicated families. We will treat the case of an odd cycle factor in Corollary~\ref{cor:odd cycle}.
	
	In the case i) of Theorem~\ref{thm:cp} that either $G$ or $H$ has a maximal general position set consisting of a pair of adjacent vertices, note that $G \cp H$ has a maximal general position set of order two contained within a single layer. This suggests the following question: can we always find a small maximal general position set of $G \cp H$ within a single $G$- or $H$-layer? A subset $S$ of a $G$-layer $G^h$ will be in general position in $G \cp H$ only if $S$ is a general position set of $G$, so any maximal general position set of $G \cp H$ belonging to a single $G$-layer must be a maximal general position set of $G$. However, that $S$ is a maximal general position set of $G$ does not necessarily imply that $S \times \{ h\} $ is a maximal general position set in the product $G \cp H$. We now present the extra property that a maximal general position set $S$ of $G$ must satisfy in order for a copy $S \times \{h\} $ of $S$ in a $G$-layer $G^h$ to be a maximal general position set of $G \cp H$.

	\begin{definition}
		A \emph{terminal set} of a graph $G$ is a maximal general position set $S$ such that for any vertex $u \in V(G) - S$ there is a shortest path of $G$ that contains $u$ as an endpoint as well as at least two vertices of $S$. The number of vertices in a largest terminal set will be denoted by $\tp (G)$ and the number of vertices in a smallest terminal set by $\tp ^- (G)$. If there is no such set, then we write $\tp ^- (G) = \tp (G) = \infty $.
	\end{definition}
	An example of a terminal position set is shown in Figure~\ref{fig:Petersen}. As a terminal set is a maximal general position set, we have $\tp (G) \geq \tp ^-(G) \geq \gp ^-(G)$ for any graph $G$. Hence for any complete graph $\tp (K_n) = \tp ^-(K_n) = n$. If $S$ is a (not necessarily maximal) general position set of $G$, then we will say that a vertex $u \in V(G)-S$ is \emph{terminal with respect to $S$}, or \emph{$S$-terminal} for short, if there is a shortest path with $u$ as an endpoint that contains two vertices of $S$; we will refer to such a path as a \emph{$(u,S)$-bad path}.
	
	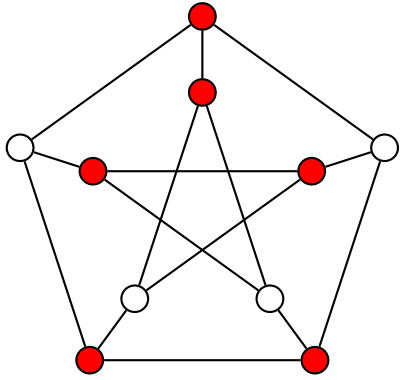
\begin{figure}[ht!]
		\centering
		\begin{tikzpicture}[x=0.2mm,y=-0.2mm,inner sep=0.2mm,scale=0.7,thick,vertex/.style={circle,draw,minimum size=10}]
			\node at (180,200) [vertex,fill=red] (v1) {};
			\node at (8.8,324.4) [vertex] (v2) {};
			\node at (74.2,525.6) [vertex,fill=red] (v3) {};
			\node at (285.8,525.6) [vertex,fill=red] (v4) {};
			\node at (351.2,324.4) [vertex] (v5) {};
			\node at (180,272) [vertex,fill=red] (v6) {};
			\node at (116.5,467.4) [vertex] (v7) {};
			\node at (282.7,346.6) [vertex,fill=red] (v8) {};
			\node at (77.3,346.6) [vertex,fill=red] (v9) {};
			\node at (243.5,467.4) [vertex] (v10) {};

			\path
			(v1) edge (v2)
			(v1) edge (v5)
			(v2) edge (v3)
			(v3) edge (v4)
			(v4) edge (v5)
			
			(v6) edge (v7)
			(v6) edge (v10)
			(v7) edge (v8)
			(v8) edge (v9)
			(v9) edge (v10)
			
			(v1) edge (v6)
			(v2) edge (v9)
			(v3) edge (v7)
			(v4) edge (v10)
			(v5) edge (v8)

			;
		\end{tikzpicture}
		\caption{The Petersen graph with a terminal position set (red). Any vertex outside the set of red vertices is the endpoint of a shortest path containing two red vertices.}
		\label{fig:Petersen}
	\end{figure}

	\begin{lemma}\label{lem:terminal upper bound}
		Let $G$ and $H$ be graphs with order at least two. If $S$ is a maximal general position set of $G$ and $h \in V(H)$, then $S \times \{ h\} $ is a maximal general position set of $G \cp H$ if and only if $S$ is a terminal position set of $G$ (with a similar result for the factor $H$). Hence, for any Cartesian product 
		\[ \gp ^-(G \cp H) \leq \min \{ \tp ^- (G),\tp ^- (H)\} .\]
	\end{lemma}
	\begin{proof}
		Firstly, suppose that $S$ is a maximal general position set of $G$, but that there is a vertex $w \in V(G)-S$ such that there is no $(w,S)$-bad path. Then for any vertex $h \in V(H)$ the set $S \times \{ h\} $ is not a maximal general position set of $G \cp H$, as for any $h' \in V(H) -\{ h\} $ the vertex $(w,h')$ can be added to $S \times \{ h\} $ without violating the no-three-in-line property. 
		
		Conversely, let $S$ be a terminal set in $G$ and consider the set $S^{\prime } = S \times \{ v\} $ for any $v \in V(H)$. By definition, $S$ is a maximal general position set of $G$, so $S^{\prime }$ is in general position in $G \cp H$ and adding a further vertex from the layer $G^v $ would create three-in-a-line. Suppose that we add a vertex $(u^{\prime },v^{\prime })$ to $S^{\prime }$, where $v^{\prime } \neq v$. Let $P$ be a shortest $v,v^{\prime }$-path in $H$. If $u^{\prime } \not \in S$, then let $Q$ be a $(u^{\prime },S)$-bad path, i.e. $Q$ is a shortest path with $u^{\prime }$ as an endpoint and containing two vertices of $S$; otherwise, if $u^{\prime } \in S$, $Q$ can be any shortest path from $u^{\prime }$ to a vertex of $S$. Now the concatenated path $_{u^{\prime}}P,Q_v$ is a shortest path in $G \cp H$ that contains three vertices of $S^{\prime } \cup \{ (u^{\prime },v^{\prime })\} $. Thus $S^{\prime }$ is a maximal general position set of $G \cp H$ and $\gp ^-(G \cp H) \leq \tp ^- (G)$ and similarly $\gp ^-(G \cp H) \leq \tp ^- (H)$.
	\end{proof}
	This allows us to immediately settle the question of the lower general position problem when one of the factors is a cycle or a wheel. 
	\begin{corollary}\label{cor:odd cycle}
		For any graph $H$ and any odd $r$, $\gp ^-(C_r \cp H) = 3$, unless $H$ has a universal line consisting of two adjacent vertices.
	\end{corollary}
	\begin{proof}
		By Theorem~\ref{thm:cp} the product $C_r \cp H$ will not have a universal line unless $H$ has a maximal general position set consisting of two adjacent vertices, so in all other cases $\gp ^-(C_r \cp H) \geq 3$. However a pair of adjacent vertices in $C_r$, together with the corresponding antipodal vertex, forms a terminal set in $C_r$, so by Lemma~\ref{lem:terminal upper bound} we obtain the upper bound $\gp ^-(C_r \cp H) \leq 3$, establishing the result.
	\end{proof}
	Products with one factor a wheel graph can be dealt with in a similar fashion, since $\g (W_n) > 2$ for $n \geq 6$. In this case any triangle is a terminal set.
	\begin{corollary}
		For $n \geq 6$, $\gp ^-(W_n \cp H) = 3$, unless $H$ has a maximal general position set consisting of two adjacent vertices.
	\end{corollary}
	We now show that Lemma~\ref{lem:terminal upper bound} is tight for products $K_{n_1} \cp K_{n_2} \cp \cdots \cp K_{n_k}$ of an arbitrary number of complete graphs (of which a rook graph $K_{n_1} \cp K_{n_2}$ is a particular instance). In the article~\cite{Ghorbani-2021} it was shown that $\gp (K_{n_1} \cp K_{n_2} \cp \cdots \cp K_{n_k}) \geq n_1+n_2+\dots +n_k-k$. 
	
	\begin{theorem}
		The lower general position number of the product of complete graphs is given by \[ \gp ^-(K_{n_1} \cp K_{n_2} \cp \cdots \cp K_{n_k}) = \min \{ n_1,n_2,\dots ,n_k\} .\]
	\end{theorem}
	\begin{proof}
		Suppose without loss of generality that $\min \{ n_1,n_2,\dots ,n_k\} = n_k$. By Lemma~\ref{lem:terminal upper bound} it follows that \[ \gp ^-(K_{n_1} \cp K_{n_2} \cp \cdots \cp K_{n_k}) \leq n_k .\] 
		Suppose that $S$ is a general position set of $K_{n_1} \cp K_{n_2} \cp \cdots \cp K_{n_k}$ with order strictly less than $n_k$. As $|S| < n_k$, there exists a vertex $(u_1,\dots ,u_k)$ such that for any $(v_1,\dots ,v_k) \in S$ we have $v_i \neq u_i$ for $1 \leq i \leq k$. Therefore the distance of $(u_1,\dots ,u_k)$ from every vertex of $S$ is $k = \diam (K_{n_1} \cp K_{n_2} \cp \cdots \cp K_{n_k})$. Thus the vertex $(u_1,u_2,\dots ,u_k)$ could be added to $S$ whilst maintaining the no-three-in-line property, so that $S$ is not maximal.
	\end{proof}
	The general position sets of products of a general graph with a clique suggests the notion of `orthogonal general position sets', which may be of independent interest.
	\begin{definition}\label{def:}
		Two (not necessarily disjoint) general position sets $S_1$ and $S_2$ are \emph{orthogonal} if any shortest path starting in $S_1$ and ending in $S_2$ contains just two vertices of the multiset $S_1 \cup S_2$.
	\end{definition}
	Note that if there is a vertex in $S_1 \cap S_2$, then this would be counted twice if it occurs in a shortest path. Hence if $S_1 \cap S_2 \neq \emptyset $, then at least one of $S_1,S_2$ must contain just one vertex. A couple of examples of orthogonal general position sets are shown in Figure~\ref{fig:orthogonal gp sets}.
	
	\begin{figure}
		\centering
		\begin{tikzpicture}[middlearrow=stealth,x=0.2mm,y=-0.2mm,inner sep=0.1mm,scale=1.0,
			thick,vertex/.style={circle,draw,minimum size=15,font=\small},every label/.style={font=\scriptsize}]

			\node at (0,0) [vertex,fill=red] (x0) {};
			\node at (-50,86.6) [vertex,fill=blue] (x1) {};
			\node at (-150,86.6) [vertex] (x2) {};
			\node at (-200,0) [vertex,fill=red] (x3) {};
			\node at (-150,-86.6) [vertex] (x4) {};
			\node at (-50,-86.6) [vertex,fill=blue] (x5) {};
			
			\node at (200,0) [vertex,fill=blue] (y) {};
			\node at (100,0) [vertex,fill=red] (y1) {};
			\node at (200,100) [vertex,fill=red] (y2) {};
			\node at (300,0) [vertex,fill=red] (y3) {};
			\node at (200,-100) [vertex,fill=red] (y4) {};
			
			\path
			
			(y) edge (y1)
			(y) edge (y2)
			(y) edge (y3)
			(y) edge (y4)
			
			(x0) edge (x1)
			(x1) edge (x2)
			(x2) edge (x3)
			(x3) edge (x4)
			(x4) edge (x5)
			(x5) edge (x0)

			;
		\end{tikzpicture}
		\caption{Examples of orthogonal general position sets in a cycle and a star (one set in red and the other in blue).}
		\label{fig:orthogonal gp sets}
	\end{figure}
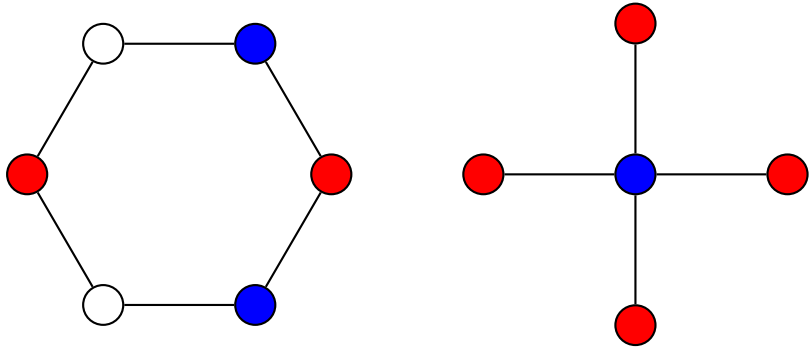

	It is easily seen that there is a one-to-one correspondence between maximal collections of $r$ orthogonal general position sets of $G$ and the $r$ layers of maximal general position sets in $G \cp K_r$; we omit the simple proof.
	\begin{lemma}
		A subset $S \subseteq V(G \cp K_r)$ is in general position if and only if the projections onto $G$ of the intersections of $S$ with the $G$-layers form a collection of orthogonal general position sets of $G$.
	\end{lemma}
	
	In general however, the lower general position number of a product $G \cp H$ and the bound of Lemma~\ref{lem:terminal upper bound} can be arbitrarily far apart.
	
	\begin{proposition}
		For any $r \geq 2$, there exist graphs $G$ and $H$ such that $G \cp H$ has a universal line and $\min \{ \tp ^- (G) ,\tp ^- (H)\} = r$.
	\end{proposition}
	\begin{proof}
		For a given $r \geq 2$, take $G$ and $H$ to be $K_{r+1}^-$, i.e. a complete graph $K_{r+1}$ with one edge deleted. If the deleted edge is $uv$, then $\{ u,v\} $ induces a universal line in $K_{r+1}^-$; hence by condition ii) of Theorem~\ref{thm:cp} we know that $\gp ^-(K_{r+1}^- \cp K_{r+1}^-) = 2$, whilst it is easily seen that the smallest terminal sets in $K_{r+1}^-$ have order $r$.  
	\end{proof}
	
	We close with a conjectured lower bound for $\gp ^-(G \cp H)$; this has been verified computationally for all pairs $G,H$, where $G$ and $H$ both have order at most six~\cite{Erskine}.

	\begin{conjecture}
		For any graphs $G$ and $H$, \[ \gp ^-(G \cp H) \geq \min \{ \gp ^-(G),\gp ^-(H)\} .\]
	\end{conjecture}

	\section{Existence of terminal sets}\label{sec:existence}
	
	The bound of Lemma~\ref{lem:terminal upper bound} raises the following important question: does every graph have a terminal set? This question turns out to be quite deep. In the following two theorems we give a constructive proof of the existence of a terminal set for any graph with diameter at most three.

	\begin{theorem}\label{thm:diam two}
		Every graph with diameter two has a terminal set.
	\end{theorem}
	\begin{proof}
		Note that a subset $S \subset V(G)$ of a graph $G$ with diameter two is in general position if and only if it is an independent union of cliques. We construct such a set in a greedy way. We start with a maximum clique $W_1$. Then for $W_2$ we take the largest clique that has no edges to $W_1$ and, in general, we take $W_i$ to be the largest clique that has no edges to $W_1 \cup W_2 \cup \dots \cup W_{i-1}$. When this process terminates, we are left with a maximal independent union of cliques $W$.
		
		It remains to show that any vertex $x \notin W$ is the endpoint of a shortest path containing two vertices of $W$. As $x$ was not selected by the algorithm, it must have an edge to some vertex of $W$. Let $r$ be the smallest value such that $x$ has an edge to $W_r$. If $x$ was adjacent to every vertex of $W_r$, then, since $x$ has no edge to $W_1 \cup \dots \cup W_{r-1}$, we would have added it to $W_r$ at that stage of the greedy algorithm; thus there are $y,z\in W_r$ such that $x \sim y$ and $x \not \sim z$, so that $x,y,z$ is a shortest path containing three points of $W \cup \{ x\} $.
	\end{proof}

	\begin{theorem}\label{thm:diameter three}
		Every graph with diameter three contains a terminal set.
	\end{theorem}
	\begin{proof}
		Let $G$ be a graph with diameter three. We give an algorithm to construct a terminal set in $G$. We begin by setting $W_1$ to be any maximal clique in $G$. Once we have selected cliques $W_1,W_2,\dots ,W_{j-1}$, we add a new clique $W_j$ as follows. Set $S_{j-1} = \bigcup _{1 \leq i \leq j-1}V(W_i)$. For $1 \leq i \leq j-1$, say that a vertex $v$ is $W_i$-equidistant if $d(v,w) = d(v,w')$ for any $w,w' \in W_i$. If $v$ is not $W_i$-equidistant for some $1 \leq i \leq j-1$, then $v$ is $S_{j-1}$-terminal. Let $T_{j-1}$ be the set of vertices of $G$ that are $W_i$-equidistant for $1 \leq i \leq j-1$ such that there are no edges between $T_{j-1}$ and $S_{j-1}$ (so that every vertex of $T_{j-1}$ is at distance two or three from $S_{j-1})$. Within $T_{j-1}$, let $R$ be a maximal clique (i.e. it may not be maximal within $G$,  but no further vertices can be added to $R$ from $T_{j-1}$). 
		
		Notice that for some $i$ we could have some vertices in $R$ that are at distance two from $W_i$ and some that are at distance three from $W_i$. We refine the set $R$ so that this situation does not occur. For $i = 1,2,\dots, j-1$, if there are vertices at distance two and distance three from $W_i$, then we remove from $R$ all of the vertices at distance three from $W_i$. We set $W_j$ to be the set that results once this refining process has been completed. Note that $W_j$ will be non-empty if $R$ was non-empty, since any single vertex from $T_{j-1}$ would satisfy this condition.
		
		When $T_{j-1}$ is empty, the algorithm terminates. We claim that the resulting collection of cliques $S$ is a terminal set. Certainly $S$ is in general position, since if $W,W'$ are cliques of $S$, then either the distance between any vertex of $W$ and any vertex of $W'$ is two, or else the distance between any vertex of $W$ and any vertex of $W'$ is three. Suppose that $u \in V(G)-V(S)$. If $u$ is not $W$-equidistant for some clique $W$ of $S$, then $u$ is $S$-terminal, so we can assume that $u$ is $W$-equidistant for every clique $W$ of $S$. If there is no edge from $u$ to $S$, then $u$ could have been added as the next stage of the algorithm, a contradiction; therefore, let $r = \min \{ i: u \sim W_i, W_i \subseteq S\} $. As $u$ is $W_r$-equidistant, $W_r \cup \{ u\} $ is a clique; in particular, as $W_1$ is a maximum clique of $G$, we must have $r > 1$. 
		
		If follows that $u$ must have been deleted in the refining process during the $r$-th stage of the algorithm.  Hence there must be an $s < r$ such that $d(w,w') = 2$ and $d(w,u) = 3$ for all $w \in W_s$, $w' \in W_r$; however, this implies that $u$ is $S$-terminal, as it is the initial vertex of a shortest path of length three to $W_s$ that passes through $W_r$. It follows that $S$ is a terminal set.
	\end{proof}
	It has also been verified computationally that all graphs with order at most eleven have terminal sets~\cite{Erskine}. We make the following conjecture.
	\begin{conjecture}\label{con:terminalsets}
		Every graph has a terminal set.
	\end{conjecture}
	
	The following lemma gives some information on the properties of a hypothetical minimal graph that does not contain a terminal set. 
	
	\begin{lemma}\label{lem:simplicial cutsets}
		If $G$ does not have a terminal set, but every proper isometric subgraph of $G$ does contain a terminal set, then \begin{itemize}
			\item $G$ is connected,
			\item $G$ has no simplicial vertices,
			\item no cut-set of $G$ is a clique,
			\item $G$ is twin-free.
		\end{itemize}
	\end{lemma}
	\begin{proof}
		Let $G$ be as described. In particular, $G$ is not a clique. It is trivial that $G$ is connected, for otherwise taking the union of terminal sets of the components of $G$ would yield a terminal set of $G$. 
		
		Suppose that $G$ has a simplicial vertex $u$. The graph $G-u$ is an isometric subgraph of $G$, so by assumption has a terminal set $S$. If $S = N(u)$, then $u \cup N(u)$ is terminal in $G$. Otherwise, we now show that if we add the vertex $u$ back to $G-u$, the set $S$ remains a terminal set. If $S \cap N(u) = \emptyset $, then any vertex $v \in N(u)$ is the endpoint of a $(v,S)$-bad path and adding the edge $uv$ gives a $(u,S)$-bad path. If $N(u) \subset S$, then, as we can assume that $S \neq N(u)$, there is a vertex $w \in S-N(u)$ and any shortest $u,w$-path would be a $(u,S)$-bad path, since it passes through $N(u) \subset S$. Finally, we can suppose that $N(u) \cap S \neq \emptyset $ and $N(u)-S \neq \emptyset $. In this case any $w_2 \in N(u)-S$ is the endpoint of a $(w_2,S)$-bad path $P$; if $P$ passes through some $w_1 \in N(u) \cap S$, then replacing the initial edge $w_2w_1$ by $uw_1$ gives a $(u,S)$-bad path, whereas if $P$ does not pass through $N(u) \cap S$, then adding the edge $uw_2$ also yields a $(u,S)$-bad path. Therefore $G$ has no simplicial vertex.
		
		Now let $G$ be a minimal graph without a terminal set that has a clique $W$ that is also a cut-set. For each component $C$ of $G-W$ there is a terminal set $S_C$ in $G[C \cup W]$. If $S_C \subseteq W$, then we must have $S_C = W$. If $S_C = W$ is a terminal set of $G[C \cup W]$ for every component $C$ of $G-W$, then $W$ is a terminal set of $G$. Otherwise, fix a component $C$ of $G-W$ such that there is a terminal set $S_C$ of $G[C \cup W]$ with $S_C \cap C \neq \emptyset $.
		
		We now show that this $S_C$ is a terminal set of $G$, i.e. if $x \in V(G-S_C)$, then $x$ is the endpoint of an $(x,S_C)$-bad path. As $G[C \cup W]$ is an isometric subgraph of $G$, this is true by assumption if $x \in (C \cup W)-S_C$, so suppose that $x \in C^{\prime }$, where $C^{\prime }$ is a component of $G-W$ distinct from $C$. Consider the collection of shortest paths from $x$ to the vertices of $S_C \cap C$. If any of these paths passes through another vertex of $S_C$, in particular a vertex of $W\cap S_C$ (which may be empty), then we are done. Otherwise, let $P$ be a shortest path from $x$ to a vertex $u \in S_C \cap C$. Set $w$ to be the first vertex of $P$ in $W$ and $P^{\prime }$ to be the $x,w$ section of $P$. As $w \not \in S_C$, by assumption there is a $(w,S_C)$-bad path $Q$ in $G[C \cup W]$, which is also a $(w,S_C)$-bad path in $G$. Consider the concatenated path $P^{\prime },Q$ in $G$. The only reason that this would not be an $(x,S_C)$-bad path in $G$ is if the first edge of $Q$ is $w,w^{\prime } $, where $w^{\prime } \in W$ and $d_G(x,w) = d_G(x,w^{\prime })$; in this case, we can replace $P^{\prime } $ by a shortest $x,w^{\prime }$-path and obtain the concatenated $(x,S_C)$-bad path $P_1,Q$.
		
		Suppose finally that $G$ contains twins $x$ and $y$. Then $G - x$ is an isometric subgraph and hence has a terminal set $S$ by assumption. If $y \not \in S$, then $y$ is the endpoint of a $(y,S)$-bad path $P$, and changing the initial point of $P$ to $x$ shows that $x$ would also be the endpoint of an $(x,S)$-bad path, so that $S$ is terminal in $G$. If $y \in S$, but $S \cap N(y) = \emptyset $, then $S \cup \{ x\} $ is terminal in $G$. Finally, if $y \in S$ and $S \cap N(y) \neq \emptyset $, then if $x \not \sim y$, $x$ would be the endpoint of a shortest path to $y$ via $S \cap N(y)$ and so $x$ would be terminal with respect to $S$ in $G$, whilst if $x \sim y$, then $S \cup \{ x\} $ would be a terminal set in $G$. 
	\end{proof}

	The existence of a terminal set in any cograph $G$ follows from Theorem~\ref{thm:diam two} by taking the union of terminal sets of the components of $G$, each of which has diameter at most two. Lemma~\ref{lem:simplicial cutsets} gives an alternative inductive proof for cographs.
	\begin{corollary}\label{cor:cograph}
		Every cograph has a terminal set.
	\end{corollary}
	\begin{proof}
		We perform induction on the order of the cograph $G$, beginning with $K_1$. Every cograph contains a pair of twins $x,y$ and by the process described in Lemma~\ref{lem:simplicial cutsets} combined with induction we can obtain a terminal set in $G$. 
	\end{proof}
	More significantly, Lemma~\ref{lem:simplicial cutsets} also shows that any chordal graph has a terminal set.
	\begin{corollary}\label{cor:chordal}
		Every chordal graph has a terminal set.  
	\end{corollary}
	\begin{proof}
		We prove the result by induction on the order. The basis follows trivially. For the induction step, let $G$ be a chordal graph with order $n$. The graph $G$ contains a simplicial vertex $u$ and $G-u$ is also chordal. As $G-u$ has a terminal set by induction, the argument of Lemma~\ref{lem:simplicial cutsets} shows that $G$ also has a terminal set.
	\end{proof}
	
	Observe that the inductive processes described in Corollaries~\ref{cor:chordal} and~\ref{cor:cograph} can also be converted into polynomial time algorithms to construct a terminal set in any chordal graph or cograph. It is an interesting question how far away the orders of the sets produced by these algorithms and those in Theorems~\ref{thm:diam two} and~\ref{thm:diameter three} can be from the lower terminal position number.

	\section{Values of the terminal position number}\label{sec:terminal pos num}
	
	In this final section we determine the terminal and lower terminal position numbers of some common graph families, namely Kneser graphs, line graphs of complete graphs and complete multipartite graphs. These graphs have diameter two (for sufficiently large order $n$), so the existence of a terminal set is guaranteed by Theorem~\ref{thm:diam two}. We begin with a lemma on the structure of terminal sets in diameter two graphs.

	\begin{lemma}\label{lem:diam2indset}
		In a diameter two graph, no terminal set is an independent set.
	\end{lemma}
	\begin{proof}
		Suppose that $S$ is an independent set that is also a maximal general position set. Let $u \in V(G)-S$. Then $u$ cannot be the initial vertex of a path of length two that also contains two vertices of $S$, since this would imply that these two vertices of $S$ are adjacent.
	\end{proof}
	The Kneser graph $K(n,2)$ is the graph with vertex set equal to all subsets of order two of $\{ 1,2,\dots ,n\} $; two such subsets $A,B \subset \{ 1,2,\dots ,n\} $ are adjacent in $K(n,2)$ if and only if $A \cap B = \emptyset $. It was shown in~\cite{Ghorbani-2021} that for large enough $n$ the general position number of the Kneser graph $K(n,2)$ is $n-1$, which corresponds to a largest independent set, whereas the lower general position number is just six for $n \geq 12$~\cite{SteKlaKriTuiYer}. The terminal position number lies strictly between these two numbers for large enough $n$. Note that in this case the algorithm of Theorem~\ref{thm:diam two} gives the exact answer.
	\begin{theorem}
		The terminal and lower terminal position numbers of the Kneser graph $K(n,2)$ are given by 
		$$\tp (K(n,2)) = \tp ^-(K(n,2)) = \left\{
		\begin{array}{ll}
			6;     &  n=5, \\
			\left \lfloor \frac{n}{2} \right \rfloor ;     &  n \geq 6.
		\end{array} \right.
		$$
	\end{theorem}
	\begin{proof}
		We start with the Petersen graph $K(5,2)$. Let $S$ be a terminal set. As the Petersen graph is triangle-free, by Lemma~\ref{lem:diam2indset} $S$ induces $rK_2 \cup sK_1$, where $r \geq 1$. The general position set displayed in Figure~\ref{fig:Petersen} is terminal; this set is also a largest possible general position set. It is also shown in~\cite{SteKlaKriTuiYer} that the lower general position number of the Petersen graph is four. Suppose therefore that $S$ has order four or five. Observe that a vertex outside $S$ is the initial point of a shortest path containing two vertices of $S$ only if it is a neighbour of a vertex in one of the $r$ copies of $K_2$ induced by $S$.  If $r = 1$, then there are just four such vertices, which, together with the vertices of $S$, accounts for at most nine vertices. If $r = 2$, then it follows from the fact that $K(5,2)$ has diameter two that the neighbourhoods of the two copies of $K_2$ are identical, so this also accounts for at most nine vertices. It follows that there are at least six vertices in any terminal set of $K(5,2)$. 
		
		Now we may assume that $n \geq 6$. By Lemma~\ref{lem:diam2indset} and the discussion in~\cite{Ghorbani-2021}, it follows that there are just two types of maximal general position set $S$ to consider, a clique of order $\left \lfloor \frac{n}{2} \right \rfloor $ or a copy of $3K_2$ induced by subsets of size two of a subset of $[n]$ of order four (say all pairs of vertices from $\{ 1,2,3,4\} $). In the latter case, the vertex $\{ 5,6\} $ is adjacent to every vertex of the general position set and hence is not terminal with respect to $S$. Suppose that $S$ is a maximum clique and let $\{ a,b\} $ be any vertex outside the clique. Then there are vertices of $S$, say $\{ a,c\} $ and $\{ d,e\} $, such that $|\{ a,b\} \cap \{ a,c\} | = 1$ and $\{ a,b\} \cap \{ d,e\} = \emptyset $, so that $\{ a,b\} $ is terminal with respect to $S$.
	\end{proof}
	The general and lower general position numbers of line graphs were studied in~\cite{SteKlaKriTuiYer} and~\cite{Ghorbani-2021}; in this case we show that the terminal position number coincides with the general position number.  
	\begin{theorem}
		The terminal position number of $L(K_n)$ is 
		$$\tp (L(K_n)) = \gp (L(K_n)) = \left\{
		\begin{array}{ll}
			n;     &  \text{ if } 3|n, \\
			n-1 ;     &  \text{ otherwise}.
		\end{array} \right.
		$$
		For $n \geq 4$ the lower terminal position number is $\tp ^-(L(K_n)) = n-1$. 
	\end{theorem}
	\begin{proof}
		We will think of general position sets of $L(K_n)$ in terms of edges of the complete graph $K_n$ with vertex set $[n]$. It is shown in~\cite{Ghorbani-2021} that the general position sets of $L(K_n)$ correspond to disjoint unions of triangles and stars in $K_n$ and that $\gp (L(K_n)) = n$ if $n$ is divisible by three and $\gp (L(K_n)) = n-1$ otherwise. In order for the general position set to be maximal, none of the stars can have order three (otherwise such a star can be extended to a triangle) and every vertex of $K_n$ is contained in an edge of $K_n$ lying in $S$, with the possible exception of a single vertex of $K_n$ if the rest of $S$ consists of triangles. 
		
		If $n \equiv 0 \pmod 3$, let $S$ be the disjoint union of $\frac{n}{3}$ triangles. If $n \equiv 1 \pmod 3$, we let $S$ be the disjoint union of $\frac{n-1}{3}$ triangles, and finally if $n \equiv 2 \pmod 3$, we let $S$ be the union of $\frac{n-2}{3}$ triangles and one other edge. These sets are maximum general position sets. Now let $e$ be any edge of $K_n$ that does not belong to one of these sets. Then $e$ must have at least one endpoint $x$ in a triangle $\{ x,y,z\} $ of $S$ and $e,xy,yz$ is a shortest path in $L(K_n)$ from $e$ containing two vertices of $S$. Thus these general position sets are terminal.
		
		Now let $S$ be a smallest possible terminal set of $L(K_n)$. $S$ can contain at most one star; otherwise, observe that the edge of $K_n$ that connects the centres of two stars would not be terminal with respect to $S$. Thus $S$ consists of a collection of triangles and at most one star, so that $|S| \geq n-1$. Now for $n \geq 4$, let $S$ be the edges of a star in $K_n$ of order $n$. Then any edge of $K_n$ not lying in $S$ connects two leaves $x,y$ of the star and, if $z$ is any other leaf and $u$ is the centre of the star, then $xy,xu,zu$ is a shortest path in $L(K_n)$ from $xy$ through two other points of $S$.
	\end{proof}
	For $L(K_n)$ the output of the algorithm in Theorem~\ref{thm:diam two} is $n-1$. We conclude by examining complete multipartite graphs and their products. 
	
	\begin{theorem}\label{thm:terminal complete multipartite}
		For $r \geq 2$ and $n_1 \geq n_2 \geq \dots \geq n_r \geq 2$, the terminal position number of a complete multipartite graph $K_{n_1,n_2,\dots ,n_r}$ is given by \[ \tp (K_{n_1,n_2,\dots ,n_r}) = \tp ^-(K_{n_1,n_2,\dots ,n_r}) = r.\]
	\end{theorem}
	\begin{proof}
		A maximal general position set $S$ in $K_{n_1,n_2,\dots ,n_r}$ consists either of all of the vertices in one of the partite sets, or else contains one vertex from each of the partite sets (and thus induces a clique). If $S$ is one of the partite sets, then no vertex of $V(K_{n_1,n_2,\dots ,n_r})-S$ is terminal with respect to $S$. However, if $S$ contains one vertex from each partite set, then any vertex $u \in V(K_{n_1,n_2,\dots ,n_r})-S$ has a shortest path to the vertex of $S$ lying in the same partite set via any other vertex of $S$. Thus the terminal position sets of complete multipartite graphs are the maximum cliques. 
	\end{proof}
	Theorem~\ref{thm:terminal complete multipartite} allows us to prove a realisation result that compares the lower general position, terminal position and general position numbers.
	\begin{corollary}
		For any $2 \leq a \leq b \leq c$ there exists a graph $G$ with \[ \gp ^-(G) = a, \tp (G) = \tp ^-(G) = b, \gp (G) = c.\]
	\end{corollary}
	\begin{proof}
		By Theorem~\ref{thm:terminal complete multipartite} a complete $b$-partite graph with smallest part of size $a$ and largest part of size $c$ has the required parameters.   
	\end{proof}
	
	Theorem~\ref{thm:terminal complete multipartite} now enables us to find the lower general position numbers of a wide range of Cartesian products of complete multipartite graphs. The results of~\cite{rodriguez-2022} show that the lower general position number of the product of two complete multipartite graphs equals two if and only if either i) at least of the graphs is a complete bipartite graph, or ii) both graphs have a part with just two vertices.
	
	\begin{theorem}
		Let $G = K_{m_1,m_2,\dots ,m_r}$ and $H = K_{n_1,n_2,\dots ,n_s}$ be complete $r$- and $s$-partite graphs respectively, where $r,s \geq 2$, $m_1 \geq m_2 \geq \dots \geq m_r$, $n_1 \geq n_2 \geq \dots \geq n_s$ and $m_1,n_1 \geq 2$. Then the lower general position number of $G \cp H$ satisfies
		\[ \min \{ r,s,m_r,n_s\} \leq \gp ^-(G \cp H) \leq  \min \{ r,s,\max \{ m_r,n_s\} \}.\]
		If either $m_r = n_s$ or $\min \{ m_r,n_s\} \geq 8$, then \[ \gp ^-(G \cp H) = \min \{ r,s,m_r,n_s\} .\]
	\end{theorem}
	\begin{proof}
		Let $G$ and $H$ be as described in the statement of the theorem. Denote the parts of $G$ by $X_1,X_2,\dots ,X_r$ and the parts of $H$ by $Y_1,Y_2,\dots ,Y_s$, where $|X_i| = m_i$ and $|Y_j| = n_j$ for $1 \leq i \leq r$ and $1 \leq j \leq s$.

		We start by showing that $\gp ^-(G \cp H) \geq \min \{ r,s,m_r,n_s\} $. Suppose that $S$ is a maximal general position set of $G \cp H$ with $|S| < \min \{ r,s,m_r,n_s\} $. Then there are parts $X_a$ and $Y_b$, $1 \leq a \leq r$, $1 \leq b \leq s$, such that $S \cap (X_a \times V(H)) = S \cap (V(G) \times Y_b) = \emptyset $. Consider a vertex $(u,v)$ in $X_a \times Y_b$. As $(u,v) \notin S$, when we add $(u,v)$ to $S$ we must create three-in-a-line. 
		
		Any neighbour of $(u,v)$ must either have first coordinate in $X_a$ or second coordinate in $Y_b$, so $(u,v)$ has no neighbours in $S$. The diameter of $G \cp H$ is four and vertices $(x,y),(x^{\prime },y^{\prime })$ are distance four apart in $G \cp H$ if and only if $x,x^{\prime } $ are distinct vertices of some $X_i$ and $y,y^{\prime } $ are distinct vertices of some $Y_j$. Similarly, if $(x,y)$ and $(x^{\prime },y^{\prime })$ are at distance three, then either $x,x^{\prime }$ are in the same part of $G$ and $y,y^{\prime }$ are in different parts of $H$, or $x,x^{\prime }$ are in different parts of $G$ and $y,y^{\prime }$ are in the same part of $H$. Hence any vertex at distance three or four from $(u,v)$ has first coordinate in $X_a$ or second coordinate in $Y_b$. It follows that all vertices of $S$ are at distance two from $(u,v)$ and so $(u,v)$ is not the endpoint of a $((u,v),S)$-bad path.
		
		Hence we can assume that $(u,v)$ is the midpoint of a shortest path $Q$ of length four with both endpoints in $S$. Therefore the endpoints of $Q$ must be of the form $(x_1,y_1)$, $(x_2,y_2)$, where $x_1,x_2$ belong to some $X_c \neq X_a$ and $x_1 \neq x_2$, and likewise $y_1,y_2$ belong to some $Y_d \neq Y_b$ and $y_1 \neq y_2$. Considering the shortest paths between vertices of $S \cap (X_c \times Y_d)$, we see that $S$ cannot contain any vertex of $(V(G)-X_c) \times (V(H)-Y_d)$, and if $(x,y) \in S \cap (X_c \times Y_d)$, then $S \cap (\{ x\} \times (V(H)-Y_d)) = S \cap ((V(G)-X_c) \times \{ y\} ) = \emptyset $. 
		
		As $|S| < \min \{ r,s,m_r,n_s\} $, there exist $\tilde{x} \in X_c-\pi _1(S)$ and $\tilde{y} \in Y_d-\pi _2(S)$. The vertex $(\tilde{x},\tilde{y})$ is at distance at least three from any vertex of $S$, so by the maximality of $S$ it must be the endpoint of an $((\tilde{x},\tilde{y}),S)$-bad path $P$ of length four. The other endpoint $(x',y')$ of $P$ must lie in $X_c \times Y_d$. However, $(x',y')$ has no neighbours in $S$, so $P$ cannot be bad. It follows that we could add the vertex $(\tilde{x},\tilde{y})$ to $S$ without creating three-in-a-line, contradicting the maximality of $S$. Thus $|S| \geq \min \{ r,s,m_r,n_s\} $. 
		
		We now consider the existence of maximal general position sets of $G \cp H$ near the lower bound of $\min \{ r,s,m_r,n_s\} $. We will write $X_r = \{ x_1,x_2,\dots ,x_{m_r}\} $ and $Y_s = \{ y_1,y_2,\dots, y_{n_s}\} $ and will assume without loss of generality that $n_s \geq m_r$. We know from Lemma~\ref{lem:terminal upper bound} and Theorem~\ref{thm:terminal complete multipartite} that $\gp ^-(G \cp H) \leq \min \{ r,s\} $. The set \[ S = \{ (x_1,y_1),(x_2,y_2),\dots ,(x_{m_r},y_{m_r}),(x_{m_r},y_{m_r+1}),\dots ,(x_{m_r},y_{n_s})\} \] is a maximal general position set of order $n_s = \max \{ m_r,n_s\} $, so $\gp ^-(G \cp H) \leq \min \{ r,s,\max \{ m_r,n_s\} \}$. Hence if $m_r = n_s$, then $\gp ^-(G \cp H) = \min \{ r,s,m_r,n_s\} $. 
		
		Now suppose that $m_r = \min \{ m_r,n_s\} \geq 8$. Let $Y,Y'$ be distinct parts of $H$ and choose vertices $y_1,y_2 \in V(Y)$ and $y_3,y_4 \in V(Y')$. Then, setting $X' = X_r - \{ x_1,x_2,x_3,x_4,x_5,x_6\} $, it is easily seen that the set \[ S = \{ (x_1,y_1),(x_2,y_1),(x_3,y_2),(x_4,y_2),(x_5,y_3),(x_6,y_3)\} \cup (X' \times \{ y_4\} ) \] is a maximal general position set of order $m_r$. Hence if $\min \{ m_r,n_s\} \geq 8$ we have the equality $\gp ^-(G \cp H) = \min \{ r,s,m_r,n_s\} $.  
	\end{proof}

	\section*{Acknowledgements}
	The research of Eartha Kruft Welton was part of an EPSRC summer internship with the Open University (Project: General Position Problems for Graphs). The first and second authors are students at the University of Cardiff. The authors thank Grahame Erskine, Sandi Klav\v{z}ar and Ismael Yero for helpful discussions of these results.
	

\end{document}